\documentclass[12pt]{amsart}
\usepackage{pifont}
\usepackage{ifthen}
\usepackage[a4paper]{geometry}
\usepackage{ntabbing}
\usepackage{amsmath}
\usepackage{amssymb}
\usepackage{subfig}
\usepackage{color}
\usepackage{graphicx}

\title{Domino shuffling on Novak half-hexagons and Aztec half-diamonds}
\author{Eric Nordenstam}
\address{%
Fakult\"at f\"ur Mathematik, University of Vienna, 
Nordbergstra\ss e 15, 1090 Wien, Austria }
\email{eric.nordenstam@univie.ac.at}
\thanks{%
Nordenstam was supported by the Austrian Science Foundation FWF, 
grant Z130-N13.
}
\author{Benjamin Young}
\address{%
Department of Mathematics, KTH
100 88 Stockholm, Sweden
}
\email{benyoung@math.kth.se}
\thanks{
Young was supported by grant KAW 2010.0063 from the Knut and Alice Wallenberg
Foundation.
}
\date{\today}

\newcommand{\Ent}{\text{Ent}}
\newcommand{\ent}{\textup{ent}}
\newcommand{\LocEnt}[1]{\ent
\left(\frac{\partial #1}{\partial x},
\frac{\partial #1}{\partial y}\right)}

\newcommand{\area}{\textup{area}}

\newboolean{dum}
\setboolean{dum}{false}
\newcommand{\comment}[1]{\ifthenelse{\boolean{dum}}{
{\par\noindent\Huge\ding{46}} \fbox{\parbox{10cm}{#1}}\par}{}}

\newcommand{\bbC}{\mathbb{C}}
\newcommand{\bbR}{\mathbb{R}}
\newcommand{\bbZ}{\mathbb{Z}}

\newcommand{\bbN}{\mathbb{N}}

\newcommand{\AD}{\text{\textsf{AD}}}
\newcommand{\HH}{\text{\textsf{HH}}}
\newcommand{\LT}{\text{\textsf{LT}}}

\newcommand{\NILP}{\text{\textsf{NILP}}}
\newcommand{\IPP}{\text{\textsf{IPP}}}
\newcommand{\ST}{\text{\textsf{ST}}}

\newtheorem{theorem}{Theorem}[section]
\newtheorem{proposition}[theorem]{Proposition}

\newtheorem{corollary}[theorem]{Corollary}
\newtheorem{lemma}[theorem]{Lemma}
\newtheorem{definition}[theorem]{Definition}
\newtheorem{algorithm}[theorem]{Algorithm}

\begin{document}
\bibliographystyle{alpha}

\comment{
$ $Id: halfhexagon.tex 63 2011-03-25 18:07:27Z benyoung $ $
}

\maketitle

\begin{abstract}
We explore the connections between the well-studied Aztec Diamond graphs and a new family of graphs called the Half-Hexagons, discovered by Jonathan Novak.  In particular, both families of graphs have very simple domino shuffling algorithms, which turn out to be intimately related.   This connection allows us to prove an ``arctic parabola'' theorem for the Half-Hexagons as a corollary of the Arctic Circle theorem for the Aztec Diamond.
\end{abstract}
\section{Introduction}
In their groundbreaking paper~\cite{gessel-viennot} gave a method for counting families of non-intersecting lattice paths between two equinumerous sets of points.   In their first example, the paths are between 
\[(0, -x_i)\text{ and }(i, -i), \;1 \leq i \leq n\]and are composed of unit-length steps up or to the right only.  
The number of such lattice paths is a Vandermonde determinant:
\begin{equation}
\label{eqn:vandermonde}
\det \left[\binom{x_i}{j}\right]_{1 \leq i,j \leq n}
= \prod_{1 \leq i < j \leq n} \frac{x_j-x_i}{j-i}.
\end{equation}
Indeed, these ideas were already present in the works of Lindstr\"om~\cite{lindstrom} and Karlin-McGregor~\cite{karlin-mcgregor}; they are applicable far beyond the scope of these papers and form the enumerative-combinatorial cornerstone for many areas of modern mathematics.

We shall focus on a specific case of the above example of Gessel-Viennot.
Jonathan Novak pointed out to us that when $x_i = 2i$, then the number of these paths is $2^{n(n+1)/2}$, which is the same as the number of domino tilings of an Aztec diamond~\cite{eklp}; he asked us for a bijection.  We didn't find one, but we did find many amazing similarities between these two models.  Namely, they have similar \emph{domino shuffles}, and similar \emph{limit laws}.  A little further detective work turned up a family of subgraphs of the Aztec diamond, which we call \emph{Aztec half-diamonds}, whose domino shuffling algorithm is \emph{identical}, in a certain sense, to that on the half-hexagon. 
Our Aztec half-diamonds are similar but not identical to the half Aztec diamond
of~\cite{FlFo11}.

\begin{figure}
\includegraphics[height=1in]{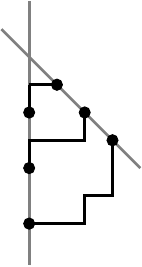} 
\includegraphics[height=1.3in, angle=90]{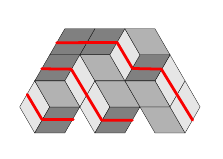} 
\!\!\!  \!\!\!  \!\!\!
\includegraphics[height=1.3in, angle=90]{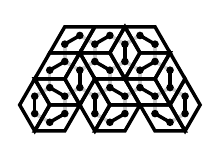} 
\!\!\!  \!\!\!  \!\!\!
\includegraphics[height=1.3in, angle=90]{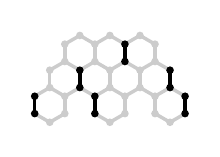} 
\!\!\!  \!\!\!  \!\!\!
\raisebox{0.6in}{
\begin{minipage}{1in}
\[
\begin{matrix}
3\\
2&5\\
1&3&6\\
1&3&5&7 
\end{matrix}
\]
\end{minipage}
}
\caption{Five models which are in bijection: Non-intersecting lattice paths, pile of boxes, perfect matching and dual tiling, interlacing particle process, Staircase semistandard Young Tableau.  All of these are order 3.
\label{fig:bijections}}
\end{figure}

\begin{figure}
\includegraphics[width=5in]{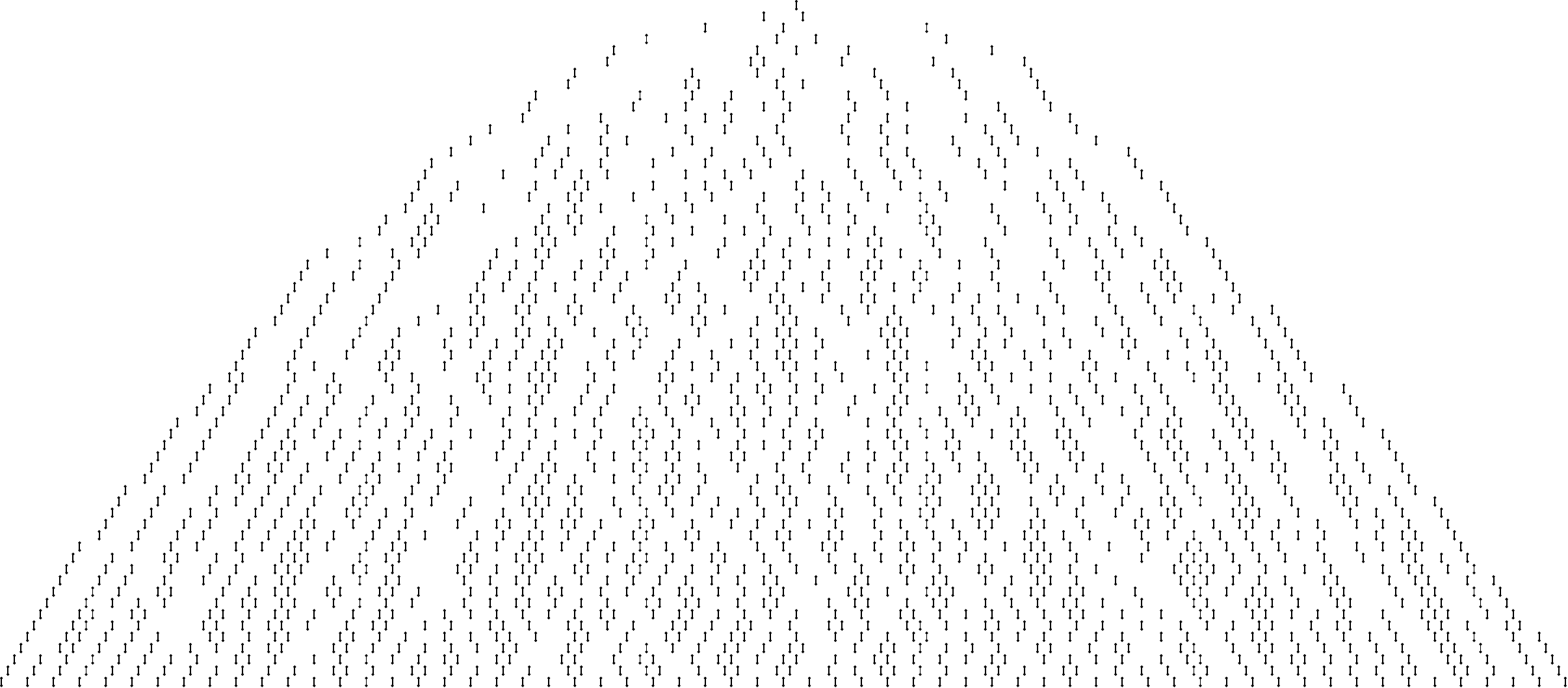}
\caption{Order 100 half-hexagon, as an interlacing particle process
\label{figorder100}}
\end{figure}

We wish to thank Jonathan Novak for bringing this problem to our attention, as well as our colleagues Alexei Borodin, Dan Romik, for helpful conversations.  This paper began at the 2010 program in Random Matrix Theory, Integrable Systems and Interacting Particle Processes at the Mathematical Sciences Research Institute in Berkeley, California; it was completed while B. Young was visiting Universitat Wien. 

Several similar problems have been considered in the vast and ever-growing literature on the dimer problem, nonintersecting lattice paths, and the like.  We mention some of them here.
\begin{itemize}
\item Okounkov and Kenyon~\cite{kenyon-okounkov} calculate limiting shapes for dimer models on portions of the hexagonal grid, under ``polygonal'' boundary conditions; they show that for a generic polygonal boundary with $3d$ sides in which the edges appear in cyclic order, the limiting shape is an algebraic curve of degree $d$.  They comment that some of these conditions can be relaxed; however it is not clear to us how to use these methods to handle the erratic bottom boundary of our half-hexagons.   
\item Di Francesco and Reshetikin~\cite{difrancesco-reshetikhin} study similar half-hexagonal shapes, but in which the long boundary is free; they obtain a variety of different limit shapes, none of which is the same as ours.
\item Borodin and Ferrari~\cite{borodin-ferrari} have a very general framework for studying dynamics on interlacing particle processes, including the Aztec diamond domino shuffle and many others.  Though they do not handle this particular case, our model does fit into this framework; we shall hopefully carry out their analysis in a subsequent paper.
\end{itemize}

\section{Bijective Combinatorics}
\label{sec:bijective combinatorics}
The Gessel-Viennot lattice path model is in bijection with a number of other combinatorial structures.  Some of these bijections are ``folklore'' and all are well-known, but it is important to state briefly what they are, in order to establish terminology.  

For the remainder of this section, fix the \emph{order} $n \in \bbZ_{\geq 0}$.

\subsection{Non-intersecting lattice paths}
Let $\NILP(n)$ be the set of families of non-intersecting lattice paths which begin at the points $(0,-2i)$ and end at the points $(i,-i)$, composed of steps of unit length in the directions of increasing $x$ and $y$.  We get this from the example of~\cite{gessel-viennot} by taking $x_i = 2i$.  
The $n$th path from the top is of length $2n$.

As mentioned above, it is easy to enumerate these families of paths using the method of Gessel-Viennot; we shall do this in Section~\ref{sec:enumerativecombinatoricslatticepath}.

\subsection{Lozenge tilings} 

It is well known (see, for example,~\cite{GeVi89,Joh05b}), that a family of 
non-intersecting lattice paths on the square lattice, with fixed start and end 
points, is in bijection with \emph{lozenge tilings} of a certain region of the 
triangular lattice.  Here, a lozenge is a parallelogram composed of two 
adjacent equilateral triangles.  The boundary of the region depends only upon 
the locations of the endpoints.  

\begin{definition}
The regular triangular lattice $\mathcal{L}$ is the infinite planar graph 
whose vertices are the integer span of the vectors
\begin{align*}
v &= \left[\begin{matrix} 1 \\ 0 \end{matrix}\right] &
w &= \frac{1}{2}\left[\begin{matrix} 1 \\ \sqrt{3} \end{matrix}\right]
\end{align*}
and which has edges joining any two vertices which are unit distance apart.  $L$ subdivides $\bbR^2$ into unit equilateral triangles.

Let $R_n \subseteq \bbR^2$ be the union of the large trapezoid with corners
\begin{align*}
\{ nv,  nw, n(w - v), -nv \} 
\end{align*}
and the $n$ small trapezoids with corners
\begin{equation*}
\left\{ \left.
\begin{matrix}
\shoveleft{(2i-n+1)v + v, } \\
\shoveleft{(2i-n+1)v - v,  }\\
\shoveleft{(2i-n+1)v - w,}  \\
\shoveleft{(2i-n+1)v + v - w}
\end{matrix}
\; \right| \;
\text{$i=0$, \dots, $n-1$} \right\}.
\end{equation*}
\end{definition}
See Figure~\ref{fig:bijections}, pictures 2 and 3, for a graphical representation of $R_3$.
Let $\LT(n)$ be the set of lozenge tilings of $R_n$.  To obtain an element of $\LT(n)$ given an element of $\NILP(n)$, first apply the affine transformation which takes the points
\begin{align*}
(0,-2i) &\longmapsto \left(-n  + 2i\right)v - \textstyle\frac{1}{2}w, \\
(i,-i) &\longmapsto -nv + \left(i - \textstyle\frac{1}{2}\right)w.
\end{align*}
Now the lattice steps are in directions $-v$ and $w-v$; each step begins and ends on the boundary of a triangle in $\mathcal{L}$ and traverses two of the triangles of $L$.  Form a partial tiling by placing the corresponding lozenge over each step.  The holes in this tiling can be covered uniquely by lozenges as well.

Observe that, in drawing this tiling, we have also drawn a pile of cubical boxes: each tile represents a visible face of a cube with faces parallel to the coordinate planes, viewed isometrically from the direction (1,1,1).

\subsection{Perfect matchings on a half-hexagon}
Let $\mathcal{L}^{\vee}$ denote the planar dual of $\mathcal{L}$. $\mathcal{L}^{\vee}$ is the regular tiling of the plane with hexagons, sometimes called the \emph{honeycomb mesh, grid} or \emph{lattice}.
\begin{definition}
Let $R^{\vee}_n$ be the subgraph of $\mathcal{L}^{\vee}$ induced by those vertices which lie within $R$.  We call $R^{\vee}_n$ the \emph{half-hexagon graph}.
Let $\HH(n)$ denote the set of perfect matchings of $R^{\vee}_n$.
\end{definition}

There is a folklore bijection between $\HH(n)$ and $\LT(n)$.  Suppose we are given an element $T$ of $\LT(n)$.  Each lozenge in the tiling $T$ is composed of two triangles of $\mathcal{L}$, which are dual to two vertices in $\mathcal{L}^\vee$.  Join every such pair of vertices with an edge to obtain a matching in $\HH(n)$.  The resulting matching is \emph{perfect} (i.e. that every vertex is covered) by virtue of the fact that a tiling in $\LT(n)$ covers $R(n)$ completely.    

\subsection{Interlacing particle process}
\label{sec:ipp}

Observe that the vertical edges of $R^{\vee}_n$ are centered at the points   

\begin{align*}
\left\{ \left.
\left(-n - \frac{1}{2} + j \right)v + (n-i)w 
\; \right| \;
\text{$i=1$, \dots, $n$;\; $j=1$, \dots, $i+n$}
\right\}.
\end{align*}

In the above set, we call $i$ the \emph{row index} and $j$ the \emph{position}. 
Let $\pi \in \HH(n)$.  Observe that the vertical edges in $\pi$ determine $\pi$ completely, subject to the following \emph{interlacing condition}:  if there are edges in positions $j$ and $j'$ of row $i$, then there must be an edge in some position $j''$ of row $i-1$, with $j \leq j'' < j$.  Indeed, any collection of vertical edges which interlace in this manner determine a tiling in $\HH(n)$.

From $\pi$, then, we may construct an \emph{interlacing particle process}
\[
  \left\{ (i,j) \;\left|\; 
(-n -\textstyle \frac{1}{2} + j )v + (n-i)w 
\text{ is the center of a vertical edge in } \pi\right.\right\} \subset \bbN^2.
\]
Let $\IPP(n)$ denote the set of all such interlacing particle processes.

\subsection{Staircase tableaux}

Let $P$ be an element of $\IPP(n)$.  Define the numbers $g_{ij}$ by 
\begin{align*}
P &= \bigcup_{i=0}^{n-1} \{ (i, g_{ij})\;|\; 0 \leq j \leq i-1 \}
\end{align*}
In other words, $g_{ij}$ is the position of the $j$th vertical edge in row $i$ of the corresponding perfect matching in $HH(n)$.  The interlacing conditions imply that
\[
g_{ij} \leq g_{i-1,j} < g_{i,j+1},
\]
in other words, that $(g_{ij})$ is a \emph{Semistandard Young tableau of staircase shape}~\cite{ECII} with bottom row equal to $(1,3,5,\ldots, 2n+1)$.  We call these objects \emph{Staircase tableaux} for short, and we write the set of all such as $\ST(n)$.  Note also that the numbers 
\[
h_{ij} = g_{ij} - j
\]
form a Gelfand-Tsetlin pattern (see~\cite[(7.37)]{ECII}) with bottom row equal to $(0,1,2,\ldots,n)$.

\section{Enumerative Combinatorics}
\label{sec:enumerativecombinatorics}
There are at least two easy ways to enumerate $\HH(n)$ by evaluating determinants.  We shall introduce a third way as a consequence of the shuffling algorithm.  Doubtless there are many others.

\subsection{Non-intersecting lattice path enumeration}
\label{sec:enumerativecombinatoricslatticepath}
Gessel-Viennot~\cite{gessel-viennot} handle a slightly more general situation. 
Observe that there are $\binom{x_i}{j}$ up-right lattice paths from the point 
$(0,-x_i)$ to $(j,-j)$.  Consequently, given $n$ integers 
$x_1 < x_2 < \cdots < x_n$, the number of families of nonintersecting lattice 
paths from $\{(0,-x_i\}$ to $\{(j,-j)\}$ is given by~\eqref{eqn:vandermonde}.
This determinant evaluation is given in~\cite{gessel-viennot}; it works because 
it is essentially a Vandermonde determinant.  Putting $x_i=2i$, we recover the 
endpoints of the paths for $\NILP(n)$, and we obtain
\[
|\NILP(n)| = 2^{n(n+1)/2}.
\]

\subsection{Staircase tableau enumeration}

An alternate easy enumeration, this time of $\ST(n)$, goes through symmetric function theory.  First we observe~\cite[(7.37)]{ECII} that because of the Gelfand-Tsetlin pattern interpretation, we have
\[
| \ST(n) | = s_\lambda(1, \ldots, 1)
\]
where $s_{\lambda}$ is a Schur function with $n+1$ arguments and $\lambda$ is the ``staircase partition'' $(1,2,\cdots,n+1)$.  It is a consequence of the classical bialternant definition of the Schur function (see~\cite[Chapter 7.15 and ex. 7.30]{ECII} that 
\[
s_{\lambda}(x_1, \ldots, x_{n+1}) = \prod_{1 \leq i < j \leq n}(x_i + x_j).
\]
Putting all $x_i$ equal to 1 gives $|\ST(n)|=2^{n(n+1)/2}$ as before.  Indeed, performing the principal specialization $x_i \mapsto q^i$ gives a $q$-enumeration of $\HH(n)$, in which each element $\pi$ of $\HH(n)$ is assigned weight proportional to $q^{\text{vol}(\pi)}$.  Here, $\text{vol}(\pi)$ denotes the integral of the \emph{height function} of $\pi$.  See~\cite{kenyon-dimer-introduction} for an introduction to the height functions of dimer models.

\section{Dynamics}
\subsection{The half-hexagon shuffle}
In this section, we will define dynamics on interlacing particles, called the \emph{half-hexagon domino shuffle}, which takes the form of a random map from $\HH(n)$ to $\HH(n+1)$.  The procedure is easiest to describe and implement on the staircase tableaux $\ST(n)$, and it is easiest to illustrate on the half-hexagon perfect matchings $\HH(n)$ (see Figure~\ref{fig:hh_shuffle}).

The half-hexagon domino shuffle constructs a perfect matching $\pi'$ on $R^{\vee}(n+1)$ from $\pi$.  It is not a deterministic process: one has to make a series of fair coin tosses to determine $\pi'$.  As it turns out, these coin tosses provide a direct explanation of the fact that the cardinality of $HH(n)$ is a power of two.

\begin{figure}
\begin{center}
\includegraphics[angle=90,height=2in]{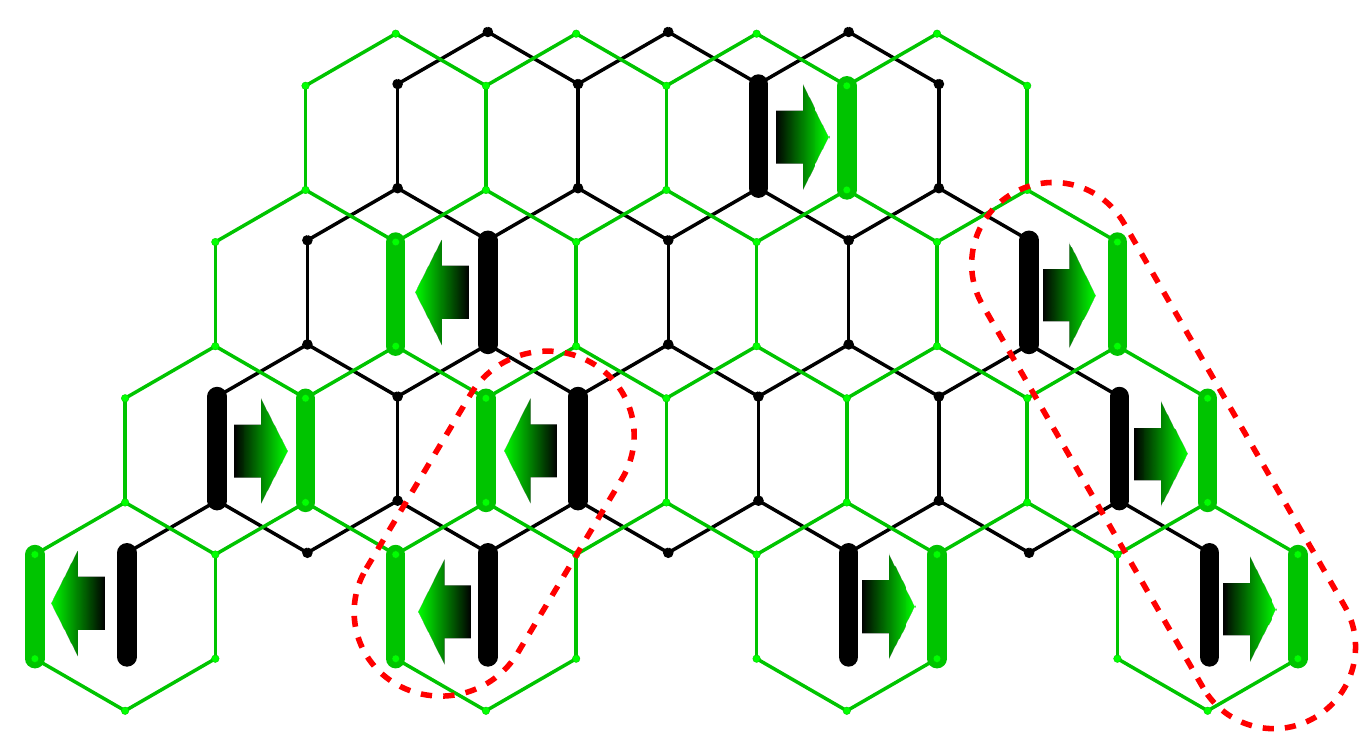}
\end{center}
\caption{Domino shuffling on the half-hexagon.  Groups of edges which are forced to move together, in order to maintain the interlacing conditions, are circled.\label{fig:hh_shuffle}}
\end{figure}

\begin{algorithm} 
\label{alg:hh_shuffle}
(Domino shuffling for the half-hexagon)
\begin{ntabbing}
123\=123\=123\=\kill
Input:  \\
\> ($g_{ij}$), a staircase tableau in $\ST(n)$, \\
\> ($\xi_{ij}$), independent Bernoulli 0-1 variables.\\ \\
Output:  \\
\> ($h_{ij}$), a staircase tableau in $\ST(n+1)$. \\ \\
\label{} for $i$ from $0$ to $n$: \\
\label{} \> for $j$ from $0$ to $i$: \\
\label{} \> \> if $j<i$ and $g_{i,j} = h_{i-1, j}$ then: \\
\label{step:forced1} \> \> \> $h_{i,j} \leftarrow g_{i, j}$    \\
\label{} \> \> else if $j>0$ and $g_{i,j} = h_{i-1, j-1}$ then: \\
\label{step:forced2} \> \> \> $h_{i,j} \leftarrow g_{i, j} + 1$ \\
\label{} \> \> else  \\
\label{step:free1} \> \> \> $h_{i,j} \leftarrow g_{i, j} + \xi_{i,j}$ \\
\label{algstep:augment1} for $j$ from 0 to n: \\
\label{algstep:augment2} \> $h_{n+1,j} = 2j + 1$
\end{ntabbing}
\end{algorithm}

It is easier to illustrate this algorithm acting on the half-hexagon (see Figure~\ref{fig:hh_shuffle}) though slightly harder to describe.
Begin with a perfect matching $\pi$ on $HH(n)$ (specified by its vertical edges).  Observe that it is possible to add, deterministically, a row of $n+1$ vertical edges, appearing every second edge, to the bottom of $\mathcal{L}^{\vee}$, in such a way as to maintain the interlacing condition.  Further, it is possible to superimpose the next-larger half-hexagon $R^{\vee}(n+1)$ on this graph, in such a way that each vertical edge is in the center of a hexagon of $R^{\vee}(n+1)$.

Working from the top of the graph to the bottom, each vertical edge in $\pi$ jumps either left or right onto the nearest vertical edge in the same row of $R^{\vee}(n+1)$.  These jumps happen independently at random, according to the result of a fair coin toss, with the following exceptions: 
\begin{itemize}
\item if moving an edge left would violate the interlacing condition with the new row above, then the edge moves \emph{right} with probability 1.
\item if moving an edge right would violate the interlacing condition with the new row above, then the edge moves \emph{left} with probability 1.
\end{itemize}

We invite the reader to check the preceding procedure is the same as Algorithm~\ref{alg:hh_shuffle}.

One does need to check that the output of Algorithm~\ref{alg:hh_shuffle} is always a staircase tableau, i.e. to verify that the interlacing conditions hold.  This is done inductively on the row $i$, together with a checking that the deterministic row $n+1$ interlaces with row $n$.

\subsection{Preserving the uniform distribution}

In this section, we argue that the domino shuffle described above preserves the uniform distribution (using the terminology of building rules).
We first introduce the \emph{time reversal} of Algorithm~\ref{alg:hh_shuffle}.

\begin{algorithm} (Time-reversed domino shuffle)
\label{alg:hh_reverse_shuffle}
(Time-reversed domino shuffling for the half-hexagon) 
\begin{ntabbing}
123\=123\=123\=\kill
Input:  \\
\> ($g_{ij}$), a staircase tableau in $\ST(n+1)$, \\
\> ($\xi_{ij}$), independent Bernoulli 0-1 variables.\\ \\
Output:  \\
\> ($h_{ij}$), a staircase tableau in $\ST(n)$. \\ \\
\label{} for $j$ from 0 to n: \\
\label{} \> $h_{n,j} = 2j+1 $ \\
\label{} for $i$ from $n-1$ down to $0$: \\
\label{} \> for $j$ from $0$ to $i$: \\
\label{} \> \> if $g_{i,j} = h_{i+1, j}$ then: \\
\label{step:forced3} \> \> \> $h_{i,j} \leftarrow g_{i, j}$    \\
\label{} \> \> else if $g_{i,j} = h_{i+1, j+1}$ then: \\
\label{step:forced4} \> \> \> $h_{i,j} \leftarrow g_{i, j}-1$ \\
\label{} \> \> else  \\
\label{step:free2} \> \> \> $h_{i,j} \leftarrow g_{i, j} - \xi_{i,j}$ \\
\end{ntabbing}
\end{algorithm}

On the half hexagon, this algorithm does the following: the $(n+1)$st row of vertical edges is dropped altogether; the $n$th row jumps deterministically to positions 0, 2, \dots, $n$.  Then, working from the bottom to the top, edges jump left or right with probability $\frac{1}{2}$, except that edges are sometimes forced to move in order to interlace with edges in the row below.

\begin{definition}
Let $\bbC \HH(n)$ denote the vector space whose orthonormal basis is indexed by the elements of $\HH(n)$.  
Let $\langle \cdot, \cdot\rangle_n$ denote the inner product which makes this an orthonormal basis.

Let $P(\pi \rightarrow \pi')$ be the probability that Algorithm~\ref{alg:hh_shuffle} produces output $\pi'$ when given input $\pi$.

Let $\psi: \bbC \HH(n) \rightarrow C\HH(n+1)$ be the linear map for which
\[ \langle \psi(\pi),  \pi' \rangle = P(\pi \rightarrow \pi'). \]
\end{definition}

\begin{lemma}
The maps $\psi$ and $\psi'$ are adjoint to each other.  That is, if $\pi \in \HH(n-1)$ and $\psi \in \HH(n)$, then
\[
\langle  \psi \pi, \pi' \rangle = \frac{1}{2^n} \langle  \pi, \psi' \pi' \rangle.
\]
\end{lemma}

\begin{proof}
Algorithms~\ref{alg:hh_shuffle} and~\ref{alg:hh_reverse_shuffle} move the edges either randomly (steps \ref{step:free1}~and~\ref{step:free2}) or deterministically (steps~
\ref{step:forced1}, 
\ref{step:forced2}, 
\ref{step:forced3}~and~
\ref{step:forced4}).  We call the random moves \emph{free} and the deterministic moves \emph{forced}, because all of the deterministic moves occur in blocks, precipitated by a preceding free move.  Figure~\ref{fig:hh_shuffle} shows an instance of the shuffling algorithm, with blocks of forced moves circled.  In this terminology, we have 
\begin{align*}
\langle  \psi \pi, \pi' \rangle &= 2^{-\#\{\text{free choices in $\pi \stackrel{\psi}{\rightarrow} \pi'$} \}} \\
\langle  \pi, \psi'\pi' \rangle &= 2^{-\#\{\text{free choices in $\pi' \stackrel{\psi'}{\rightarrow} \pi$} \}}
\end{align*}

Let $S$ be all of the \emph{free choices} in $\pi \stackrel{\psi}{\rightarrow} \pi'$; and let $S'$ be the set of all free choices which do not force any edges in row $n$ of $\pi$ to move.  Since there are $n$ edges in the bottom row, $|S| = |S'| + n$.  

Let $e_i$ be one of the edges in $E$ in row $i$, and suppose that in passing from $\pi$ to $\pi'$.  Say that $e_i$ forces the movement of edges $e_{i+1}, \ldots, e_{i'}$, in successive rows.  Observe, then, that the time-reversed algorithm $\psi'$, in passing from $\pi'$ to $\pi$, sees a free choice at edge $e_{i'}$ which forces edges $e_i, \ldots, e_{i+1}$ to move.  Moreover, $\pi'$ moves row $n$ deterministically (and then deletes it) so these are never free choices.  As such, the free choices in $\pi' \stackrel{\psi'}{\rightarrow} \pi$ are in bijection with $S'$.
\begin{equation*}
\langle \psi \pi, \pi'\rangle = 2^{-|S|} = 2^{-|S'|-n} = 2^{-n}\langle\pi, \psi'\pi'\rangle
\end{equation*}
\end{proof}

\begin{definition}
Let $\mu_n\in \bbC \HH(n)$ denote the vector corresponding to the uniform probability distribution on $\HH(n)$:  
\[
\mu_n = \frac{1}{|HH(n)|} \sum_{\pi \in \HH(n)} \pi. 
\]
\end{definition}

\begin{proposition} Domino shuffling preserves the uniform distribution:
 \[\psi \mu_{n-1} = \mu_{n}.\]
\end{proposition}

\begin{proof}
This is a straightforward computation:
\begin{align*}
\psi \mu_{n-1} 
&= 
\sum_{\pi \in \HH(n-1)} \; \sum_{\pi' \in \HH(n)}  
\langle  \psi \pi, \pi' \rangle \pi' \\
&= \frac{1}{2^n}
\sum_{\pi' \in \HH(n)} \pi' \left(\sum_{\pi \in \HH(n-1)} 
\langle  \pi, \psi'\pi' \rangle\right)  \\
&= \frac{1}{2^n} \sum_{\pi' \in HH(n)} \pi'.
\end{align*}

In the second line, the bracketed sum is equal to one for any $\pi'$ because $\psi'$ is stochastic.  The latter vector is proportional to $\mu_n$, and is thus equal to $\mu_n$ because the coordinates of both vectors sum to one.
\end{proof}

Note that it follows from this proof that 
$|HH(n)| = 2^n |HH(n-1)|$, which provides a new derivation
of the fact that  $|HH(n)| = 2^{n(n+1)/2}$.

\section{Review: the Aztec diamond}
\label{sec:aztec}

The \emph{Aztec Diamond} of order $n$ (which shall be denoted $A_n$)
is a certain shape in
the plane that can be covered with dominoes in $2^{n(n+1)/2} $ ways. 
More precisely, $A_n$ consists of the union of all squares 
whose corners have integer coordinates, whose sides are parallel to the
coordinate axes and whose interiors are contained
in $\{(x,y)\in \mathbb{R}^2 : |x|+|y| < n+1\}$.
The model was introduced in~\cite{eklp,eklp2} in the study of Alternating 
sign matrices and a good survey is~\cite{Joh05}.  
In this section we give a short review of the results we
need about this model with citations. 

One way to study the random tilings in this model is to introduce 
a certain particle process, see \cite{Joh05,Nor10} and also  
Figure~\ref{fig:aztecparticles} for details. 
In short, the lattice on which the dominos are placed is colored like a chessboard, and all dominoes whose bottommost or rightmost square is dark is 
represented by a particle.
The tiling does uniquely determine the positions of the
particles. The converse is only true modulo the fact
that there will be $2\times 2$-rectangles where two horizontal dominoes
can be switched for two vertical ones or the other way around. 
Particle configurations coming from a tiling of $A_n$ are in bijection
with Alternating sign matrices and this is the original motivation
for studying this tiling model. 
However, the measure on particle configurations that is induced by
uniform measure on all possible tilings is not the same 
as uniform measure on all Alternating sign matrices (it is, rather, connected to the \emph{2-enumeration} of alternating sign matrices, see~\cite{eklp}).

In a tiling of $A_n$ there are $\binom{n+1}2$ particles
on $n$ rows, see Figure~\ref{fig:aztecparticles}. 
Along the line $y=1$ there is a single particle, along 
line 2 there are two, etc. Let $x^{i}_j$ be the position
of the $j$th particle on line $y=i$.  Observe that the particles \emph{interlace} in the sense that
\[
x_i^{j+1} \leq x_i^j \leq x_{i+1}^{j-1};
\]
this is similar, but not the same, as the interlacing conditions for the half-hexagon particle process in Section~\ref{sec:ipp}.

There is an algorithm, called the \emph{shuffling algorithm}, 
that can be used to construct a random tiling and 
is described in great depth in each of~\cite{eklp2,JoPrSh98,Pro03}.
The procedure starts with a tiling of $A_n$; one moves the dominos of the tiling about in a certain way, decided by a certain number
of coin flips, producing a tiling of $A_{n+1}$.  
If one starts with the empty tiling of $A_0$, performs
this algorithm $n$ times, tossing fair coins all the way, 
then one ends up with a sample from the uniform distribution
of all possible tilings. 

The idea of~\cite{Nor10} is to look at the positions of the 
aforementioned particles under the evolution of this algorithm.
It turns out that the particles are not glued to the tiles.  
The dynamics of these particles are as follows.
In the following all $\gamma^i_j(t)$ for $i$, $j$, $t=1$, 2, \dots,
are independent Bernoulli random variables, that is 
one with probability $\frac12$ and zero otherwise.

\begin{figure}
\includegraphics{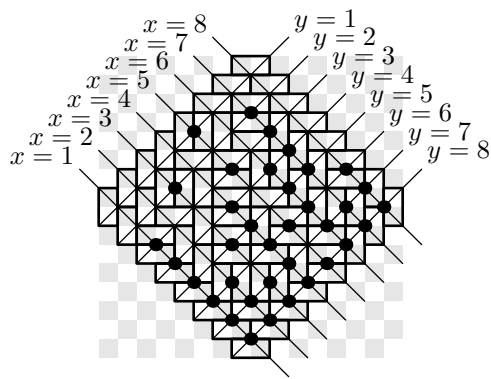}
\caption{Particle processes associated with domino tilings of
an Aztec diamond. The underlying lattice is colored like a chessboard
and all dominoes whose bottommost or rightmost square is dark is 
represented by a particle. }
\label{fig:aztecparticles}
\end{figure}

It turns out that the first particle
performs the simple random walk
\begin{equation}
x^1_1(t)=x^1_1(t-1)+\gamma_1^1(t).
\end{equation}
The particle $x^2_1$ performs a simple random walk with a reflecting
boundary. More precisely, 
while $x^2_1(t)<x^1_1(t)$ it performs a random walk independently
of  $x^1_1$, at each time either staying or adding one with 
equal probability. 
However, when there is equality, $x^2_1(t)=x^1_1(t)$, it is pushed
forward by that particle. In order to represent this as
a formula, we subtract one if the particle attempts to jump
past $x^1_1$. 
\begin{equation}
x^2_1(t)=x^2_1(t-1)+\gamma^2_1(t) 
- \mathbf{1}\{x^2_1(t-1)+\gamma^2_1(t)= x^1_1(t-1)+1\}
\end{equation}
By symmetry then 
\begin{equation}
x^2_2(t)=x^2_2(t-1)+\gamma^2_2(t)  + \mathbf{1}\{x^2_2(t-1)+\gamma^2_2(t)= x^1_1(t-1)\}.
\end{equation}
The same pattern repeats itself evermore.
\begin{align}
x^j_1(t)&=x^j_1(t-1)+\gamma^j_1(t)
- \mathbf{1}\{x^j_1(t-1)+\gamma^j_1(t)= x^{j-1}_1(t-1)+1\}\\
x^j_j(t)&=x^j_j(t-1)+\gamma^j_j(t)
+ \mathbf{1}\{x^j_j(t-1)+\gamma^j_j(t)= x^{j-1}_{j-1}(t-1)\}\\
x^j_i(t)&=x^j_i(t-1)+\gamma^j_i(t)
- \mathbf{1}\{x^j_i(t-1)+\gamma^j_i(t)= x^{j-1}_{j}(t-1)+1\}\\
&\quad\quad\quad\quad\quad\quad\quad\quad\ \;
+ \mathbf{1}\{x^j_i(t-1)+\gamma^j_i(t)= x^{j-1}_{j-1}(t-1)\}.
\end{align}
with initial conditions
$x^j_i(j)=i$
for $1\leq i \leq j$.

In the analysis~\cite{Nor10} of the asymptotics of the domino shuffling algorithm, it turns out to be quite inconvenient that the $n$ 
particles on level $n$ are not created until time $n$. A simple change of 
variables will fix this. Let
\begin{equation}
\label{eqn:capitalx} 
X_i^j (t) = x_i^j(t+i)
\end{equation}
for $1\leq i\leq j$ and $t=1$, 2, \dots.  We mention this here for a completely different reason:
Rewriting the above recursion formulas in terms of the variables
$(X^j_i)_{1\leq i\leq j}$ gives the same recursion as is implemented 
by Algorithm~\ref{alg:hh_shuffle} above (though with a different initial condition, see Section~\ref{sec:half aztec shuffle}).

\begin{figure}
\includegraphics{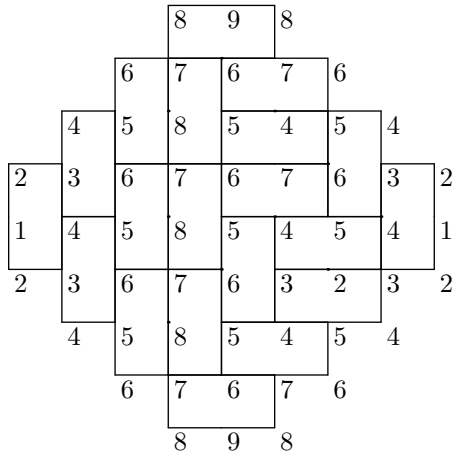}
\caption{The height function is defined on the lattice on which 
the corners of the dominoes line up.}
\end{figure}

\section{The Aztec Half-Diamond}

Recall the definition of the Aztec diamond $A_n$ in Section~\ref{sec:aztec}.
The Aztec half-diamond $H_n$ of order $n$ is a certain subregion of $A_n$. 
More precisely,  for $n$ even, 
\begin{equation*}
H_n  = \{(x,y) \in A_n :  y \geq 2 \lfloor x/2 \rfloor \} 
\end{equation*}
and, for $n$ odd,   
\begin{equation*}
H_n  = \{(x,y) \in A_n :  y + 1 \geq 2 \lfloor (x+1)/2 \rfloor \}. 
\end{equation*}
Though the definition as stated is a bit cryptic, the 
Figure~\ref{fig:half-diamond} should make this quite clear.

\begin{figure}
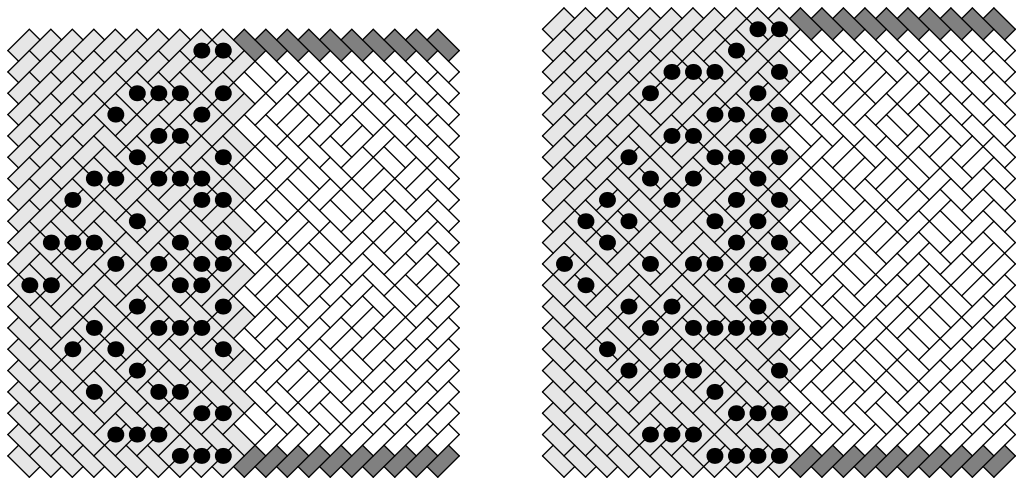

\begin{center}
\includegraphics{half_aztec20.mps}
\hspace*{0.3in}
\includegraphics{half_aztec21.mps}
\end{center}
\caption{Two Aztec half-diamonds make an Aztec diamond.  The left picture shows half-diamonds of orders 18 and 20; the right picture orders 19 and 21.  The interlacing particle process is also shown.  The right picture is the domino shuffle of the left.  
\label{fig:half-diamond}}
\end{figure}

\subsection{Dynamics: Domino shuffling}
\label{sec:half aztec shuffle}

The Aztec Half-Diamonds also have a domino shuffling algorithm.  Propp~\cite{Pro03} describes a way to generate random perfect matchings on certain planar bipartite graphs called ``generalized domino shuffling''.  The strategy is to embed the graph $G$ into an Aztec diamond of sufficiently large order $n$, and then compute a certain series of probabilities $P^{i}_{j}(m)$ where $1 \leq m \leq n$ and $1 \leq i,j \leq m$ (this is perhaps an overly concise summary, but both the manner of embedding and the means of computation are described explicitly enough in~\cite{Pro03} to allow computer implementation).  To generate the perfect matching, one then simply runs the domino shuffling algorithm, with the following modification: make $\xi^i_j(m)$ a Bernoulli random variable which takes the value 1 with probability $P^{i}_{j}(m)$.  The final tiling of the Aztec Diamond of order $n$, restricted to the embedded copy of $G$, turns out to be uniformly random; the intermediary tilings of the smaller Aztec diamonds may be discarded.

Using the embedding shown in Figure~\ref{fig:half-diamond} for generalized domino shuffling, one can check that the intermediary tilings of the smaller Aztec diamonds are also of the form of Figure~\ref{fig:half-diamond}!  As such, Propp's algorithm is a ``domino shuffle'' for the Aztec half-diamond in our sense:  a random, locally defined map which increases the size of the half-diamond but preserves the uniform distribution.  Indeed, this procedure is exactly the same as ordinary domino shuffling everywhere except the center line when $n$ is even; at these times, the particles in the center are forced to jump to equally spaced positions.  This is equivalent to imposing condition
\begin{equation*}
x^{m+1} _i (2m+1) = 2i
\end{equation*} 
 for $1\leq i\leq m $.

\subsection{Height functions and limit shape}

It is possible~\cite{Thurston} to associate a discrete surface in $\bbR^3$, called a \emph{height function}, to any domino tiling of the plane.  More precisely, the height function $h:\bbZ^2 \rightarrow \bbZ$, where the domain is the square grid whose vertices coincide with the corners of the dominos and the centers of their edges.  
\begin{definition}
Let $T$ be a domino tiling of the region $R$.  Then $h:\bbZ^2 \rightarrow \bbZ$ is a \emph{height function for $T$} if, whenever $x+y \equiv 0 \pmod{2}$ and $(x,y)$ is in $R$, then
\begin{itemize}
\item $h(x,y) = h(x,y+1)+1$ if the edge from $(x,y)$ to $(x,y+1)$ does not cross a domino in $T$;
\item $h(x,y) = h(x+1,y)-1$ if the edge from $(x,y)$ to $(x+1,y)$ does not cross a domino in $T$.
\end{itemize}
\end{definition}
Note that $h$ determines $T$ uniquely, and that two height functions for $T$ differ only by a constant.

This definition coincides with those in~\cite{Thurston, cohn-kenyon-propp} and, In the case where the region $R$ is an Aztec Diamond, this definition appears in~\cite{eklp}, where it is closely related to the height function for an alternating sign matrix. 

The reader should be advised that there are  closely related concepts in the 
literature called relative height 
functions~\cite{kenyon-okounkov, kenyon-okounkov-sheffield}, edge-placement 
probabilities and one-point functions for the particle 
process~\cite{kenyon, Pro03}.

Fix a region $R$, a tiling $T$ of $R$, and a height function $h$ for $T$.  Since no tiles ever cross the boundary of $R$, the restriction of $h$ to $\partial R$ is independent of $T$.  Indeed, $\left.h\right|_{\partial R}$ can even be computed without specifying $T$ at all.  
\begin{definition}
The function $\left.h\right.|_{\partial R}$ is called a \emph{boundary height function}. 
\end{definition}

\begin{proposition} (See \cite[Section 4]{Thurston})
$R$ possesses a domino tiling if and only if $R$ has a well-defined boundary height function.
\end{proposition}

Once it became possible to generate uniformly random tilings of large Aztec diamonds, it became immediately obvious that all such tilings have the same ``shape''.  It was first shown in~\cite{JoPrSh98} that a typical tiling of a large Aztec diamond has all of its disorder concentrated in a circular region, where the circle is tangent to all four sides of the diamond; the tiling is frozen close to the four corners.  

In fact, more is true: the height function of a random tiling of an Aztec diamond tends to the following explicit limit.  This limit seems possible to do using the correlation kernel for the Aztec Diamond, as defined in\cite{Joh05}, and something similar (asymptotics for edge placement probabilities) were computed in~\cite{CoElPr96}), but the first explicit derivation would seem to be in~\cite{Romik}.  The coordinates are slightly different:  in the following theorem, $h_{i,j}^*$ represents a rescaled height function of an order $n$ Aztec diamond, where the domain is to $[0,1] \times [0,1]$ and the range is rescaled to $[0,1]$.  Indeed, for the remainder of this section we will work in these coordinates.

\begin{theorem}  (Theorem 11' in~\cite{Romik})
\label{thm:romik_height_function}
Define
\begin{eqnarray*}
\!\!\!\!\!Z(x,y)\!\! &=& 
\frac{2}{\pi}\left[(x-1/2) \arctan\left( \frac{\sqrt{\frac14-(x-1/2)^2-(y-1/2)^2}}{1/2-y} \right) \right.
\nonumber \\ & & + \frac{1}{2} \arctan \left( \frac{2(x-1/2)(1/2-y)}{\sqrt{\frac14-(x-1/2)^2-(y-1/2)^2}} \right)
\label{eq:def-zxy} \\ & & \left. - (1/2-y) \arctan \left( \frac{x-1/2}{\sqrt{\frac14-(x-1/2)^2-(y-1/2)^2}} \right) \right].
\nonumber
\end{eqnarray*}
and define
\[
G[x,y] = \begin{cases}
x+y & 0 \le x \le \frac{1-2\sqrt{y(1-y)}}{2}, \\
x+ Z(x,y) & \frac{1-2\sqrt{y(1-y)}}{2} < x < \frac{1+2\sqrt{y(1-y)}}{2}, \\
x-y & \frac{1+2\sqrt{y(1-y)}}{2} \le x \le 1. 
\end{cases}
\]

Then as $n\to\infty$ we have the convergence in probability
$$
\max_{0 \le i,j \le n} \left|\frac{{h_{i,j}^*}^n}{n} - G(i/n,j/n) \right| \xrightarrow[n\to\infty]{\mathbb{P}}
0.
$$
\end{theorem}

In comments after Equation (15), \cite{Romik} observes that 
\begin{equation}
\label{eqn:flat_in_the_middle}
G[x,1/2] = \frac{1}{2};
\end{equation}  
that is to say, that the limiting height function is constant across the center of the Aztec Diamond.  (It is also possible to observe this directly: the Aztec diamond has a reflection symmetry in the line $y=1/2$ which leaves the uniform measure invariant, but which negates height functions up to a constant).
This observation is quite important for our purposes: it means that the limit shape for Aztec Diamonds coincides with the \emph{boundary height function} for Aztec Half-Diamonds. 

If we replace the Aztec diamonds with a sequence of regions $R_n$ approximating a given region $R$, in such a way that the rescaled boundary height functions also tend to a limiting boundary height function on $R$, then there is still a unique limiting shape and a variational principle (though in general there is no reason to expect this limiting shape to be as nice as a circle!)

\begin{theorem} (\cite[Theorem 1.1]{cohn-kenyon-propp}
\label{thm:variational-principle}
Let $R^*$ be a region in $\bbR^2$ bounded by a piecewise smooth, simple
closed curve $\partial R^*$.  Let $h_b : \partial R^* \rightarrow \bbR$ be a
function which can
be extended to a function on $R^*$ with Lipschitz constant at most $2$
in the sup norm.  Let $f : R^* \rightarrow \bbR$ be the unique such
Lipschitz function
maximizing the entropy functional $\Ent(f)$,
subject to $f |_{\partial R^*} = h_b$.

Let $R$ be a lattice region that approximates $R^*$ when rescaled by a
factor of $1/n$, and whose normalized boundary height function
approximates $h_b$.  Then the normalized height function of a random
tiling of $R$ approximates $f$, with probability tending to $1$ as $n
\rightarrow \infty$.

\end{theorem}

In the above theorem, $\Ent(f)$ is given by
\begin{equation}
\label{eqn:entropy_of_region}
\Ent(h) = \frac{1}{\area(R^*)}\iint_{R^*} \LocEnt{h} \, dx \,
dy,
\end{equation}
and $\LocEnt{h}$ is a certain explicit function of the gradient of the surface $h$; see~\cite{cohn-kenyon-propp} for further details.  For us, the most important point is that the asymptotic number of tilings depends only upon the asymptotic height function $g$, not upon the precise nature of the boundary of $R$; moreover this dependence is local.
As such, an immediate consequence of Theorem~\ref{thm:variational-principle} is the following: 

\begin{lemma}
\label{lem:subregion_height_function}
Let $S \subseteq R$ be regions.  Let $h_{\partial R}$ be a boundary height function on $\partial R$, and let $f:R \rightarrow \bbR$ be the entropy-maximizing asymptotic height function for $(R, h_{\partial R})$.  Let $g$ be the entropy-maximizing asymptotic height function for $(S, f|_{\partial S})$.  Then $g = f|_{S}$.
\end{lemma}

\begin{proof}
We will write $\Ent_S(.), \Ent_R(.), \Ent_{R\setminus S}(.)$ for the entropy functionals of the various regions, and $A_{S}, A_{R}, A_{R\setminus S}$ for their areas.  It is a direct consequence of~\eqref{eqn:entropy_of_region} that
\[
A_R \Ent_R(f)  = A_S \Ent_S(f) + A_{R\setminus S} \Ent_{R \setminus S}(f).
\]
Define a height function $g^*$ on $R$ as follows:
\begin{equation}
g^{*}(x) = 
\begin{cases}
g(x) & \text{ if }x \in S \\
f(x) & \text{ if }x \in R \setminus S
\end{cases}
\end{equation}
Observe that this function is a well-defined asymptotic height function because $g = f$ on the boundary of $S$.
If $\Ent_S(f|_{\partial S))} < \Ent_S(g)$, then $\Ent_R f < \Ent_R g^*$, contradicting the fact that $f$ is the entropy maximizer on $R$.  Thus $\Ent_S(f|_{\partial S))} = \Ent_S(g)$, and so $g = f|_{S}$ is the unique entropy maximizer on $S$. 
\end{proof}

\begin{corollary}
The limit shape of the Aztec half-diamond is the restriction of $G$ (from Theorem~\ref{thm:romik_height_function}) to $y<\frac{1}{2}$.
\end{corollary}

\begin{proof}
Let 
\begin{align*}
R&=[0,1] \times [0,1] &
f(x,y)&=G[x,y] \\
S&=[0,1] \times [0,1/2] &
g|_{\partial S} = f|_{\partial S};
\end{align*}
That is, $R$ and $f$ describe the limit height function for the Aztec Diamond, and $S$ is the limiting region for the Aztec Half-Diamond, and $g|_{\partial S}$ is its limiting height function by Equation~\eqref{eqn:flat_in_the_middle}.  Lemma~\ref{lem:subregion_height_function} now says that the limiting height function for the Aztec Half-Diamond $g$ coincides with $f_{S}$.
\end{proof}

\begin{figure}

\begin{center}
\includegraphics[width=5.5in]{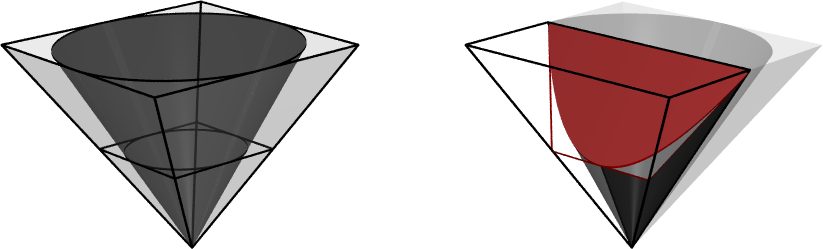}
\end{center}
\caption{Space-time diagram for domino shuffling.  
Left image: the Aztec diamonds and their arctic circles.  
Right image: the half-hexagon and its arctic parabola.  
\label{fig:cone}}
\end{figure}

\begin{corollary}
The limit shape of the half-hexagon is a parabola.
\end{corollary}

\begin{proof}
One checks that the particle process defined by the half-hexagon domino shuffle has the Aztec Half-Diamonds as its constant-time slices (see Figure~\ref{fig:cone}).  Indeed, the behaviour of Algorithm~\ref{alg:hh_shuffle} coincides with the equations~\ref{eqn:capitalx}, except that the initial condition is different;  it straightforward to check, that this initial condition is enforced both by Algorithm~\ref{alg:hh_shuffle} and by Propp's generalized domino shuffling~\cite{Pro03} applied to the half-hexagon; see Section~\ref{sec:half aztec shuffle}.

The space-time diagram of domino shuffling traces out a cone inscribed in a square pyramid.  The half-hexagons' perfect matchings correspond to slices parallel to one of the sides of this cone; the intersection is a parabola.

\end{proof}

We remark that one can in fact write down the limiting height function in this way as well; it is precisely the image of the height function of Theorem~\ref{thm:romik_height_function} under the affine transformation which takes the rectangle $[0,1] \times [0,1/2]$ to the trapezoid with corners $\{(\pm1,0), (\pm(1/2), \sqrt{3}/2)\}$.  Moreover, domino shuffling on the Aztec half-diamond has a reasonably nice description: it turns out to coincide with a particular instance of Propp's Generalized Domino shuffling~\cite{Pro03} up to gauge transformation.

\section{Future work and open problems}

The reader will doubtless have noticed that this paper is a study of the phenomenology of one of the nicest, most special cases of the dimer model on planar bipartite graphs; naturally one should try to push the analysis to more general situations.  For instance, a few other domino shuffles have been discovered on a variety of statistical mechanical models (see, for example,~\cite{CoYo10, BoGo09}) and it would be very interesting to study different space-time sections of them.  

Taking slices through the cone in Figure~\ref{fig:cone} with varying slopes should yield a family of particle processes whose limit shapes are conic sections.  Though this is a fairly trivial observation, we note that relatively few instances of the dimer model have known low-degree algebraic limit shapes.

There is a general framework~\cite{borodin-ferrari} for studying dynamics of the sort we have described in this paper, geared in particular to asymptotics.  It appears that our model fits into this framework, but we have yet to determine whether the computations are tractable.

We conclude with the following remark: the problem which Jonathan Novak posed to us --- give a bijection between $\AD(n)$ and $\HH(n)$ --- is still open, and we would be very happy to see a solution!  

\clearpage
\bibliography{references}

\end{document}